\documentclass[a4paper,11pt]{amsart}
\usepackage[utf8]{inputenc}
\usepackage[american]{babel}

\usepackage{amsmath,amsfonts,amssymb,amsthm}  
\usepackage{mathbbol}
\usepackage{accents}
\usepackage{mathrsfs}
\usepackage{mathtools}

\usepackage[x11names,rgb,table]{xcolor}
\usepackage[naturalnames=true]{hyperref}
\hypersetup{
  colorlinks = true,
  urlcolor = DeepSkyBlue4, linkcolor = Chartreuse4, citecolor = DarkOrange2
}

\usepackage{fixmath}
\usepackage{paralist}
\usepackage{stmaryrd}

\usepackage{booktabs}
\usepackage{array}
\usepackage{csquotes}

\usepackage{pgf,tikz}
\usepackage{pgfplots}
\usetikzlibrary{arrows}
\usetikzlibrary{patterns}
\usetikzlibrary[positioning]
\usetikzlibrary[fit]
\usetikzlibrary{shapes.geometric}
\usetikzlibrary{calc}

\newtheorem{theorem}{Theorem}
\newtheorem*{theorem*}{Theorem}
\newtheorem{proposition}[theorem]{Proposition}

\newtheorem{lemma}[theorem]{Lemma}
\theoremstyle{remark}
\newtheorem{remark}[theorem]{Remark}

\newtheorem{example}[theorem]{Example}
\newtheorem{question}[theorem]{Question}

\newcommand\CC{{\mathbb C}}
\newcommand\FF{{\mathbb F}}
\newcommand\KK{{\mathbb K}}
\newcommand\LL{{\mathbb L}}

\newcommand\QQ{{\mathbb Q}}
\newcommand\RR{{\mathbb R}}

\newcommand\Sph{{\mathbb S}}

\newcommand\cC{{\mathcal C}}

\newcommand\1{{\mathbb 1}}
\newcommand\koebe{Q}

 \newcommand\norm[1]{|\hskip-.2ex|#1|\hskip-.2ex|}
\newcommand\SetOf[2]{\left\{#1 \mid #2\right\}}
\newcommand\smallSetOf[2]{\{#1 \mid #2\}}

\newcommand\transpose[1]{{#1}^{\top}}

\newcommand\colvectwo[2]{\begin{pmatrix} #1 \\ #2 \end{pmatrix}}

\DeclareMathOperator\GL{GL}
\DeclareMathOperator\PGL{PGL}
\DeclareMathOperator\PSL{PSL}
\DeclareMathOperator\SL{SL}

\DeclareMathOperator\SO{SO}

\DeclareMathOperator\Orth{O}

\DeclareMathOperator\diag{diag}
\DeclareMathOperator\vol{vol}
\DeclareMathOperator\conv{conv}

\newenvironment{dedication}
  {\vspace*{0.2cm}
   \itshape             
   \raggedright         
  }
  {\par 
   \vspace{0.2cm} 
  }

\definecolor{amber}{rgb}{1.0, 0.75, 0.0}
\definecolor{brightube}{rgb}{0.82, 0.62, 0.91}
\setdefaultitem{$\triangleright$}{}{}{}

\makeatletter
\@namedef{subjclassname@2020}{%
  \textup{2020} Mathematics Subject Classification}
\makeatother

\title{Algebraic degrees of $3$-dimensional polytopes}

\author{Mara Belotti \and Michael Joswig \and Marta Panizzut} 

\address[Mara Belotti, Marta Panizzut]{
  Technische Universität Berlin,
  Chair of Discrete Mathematics/Geometry \\
  \texttt{\{belotti,panizzut\}@math.tu-berlin.de}
}

\address[Michael Joswig]{
  Technische Universität Berlin,
  Chair of Discrete Mathematics/Geometry \\
  Max-Planck Institute for Mathematics in the Sciences, Leipzig \\
  \texttt{joswig@math.tu-berlin.de}
}

\thanks{This research has been funded by the Deutsche Forschungsgemeinschaft (DFG, German Research Foundation) under Germany's Excellence Strategy – The Berlin Mathematics Research Center MATH$^+$ (EXC-2046/1, project ID 390685689). M. Joswig is also founded by the Deutsche Forschungsgemeinschaft (DFG, German Research Foundation) – Project-ID 286237555 – TRR 195, and supported by the GRK 2434: \enquote{Facets of Complexity}.}

\begin{document}

\begin{abstract}
  Results of Koebe (1936), Schramm (1992), and Springborn (2005) yield realizations of $3$-polytopes with edges tangent to the unit sphere.
  Here we study the algebraic degrees of such realizations.
  This initiates the research on constrained realization spaces of polytopes.
\end{abstract}

\maketitle
\begin{dedication}
Dedicated to Bernd Sturmfels on the occasion of his 60th birthday.
\end{dedication}
\section{Introduction} 
\noindent
Two (convex) polytopes are \emph{combinatorially equivalent} if their face lattices are isomorphic.
The \emph{realization space} of a polytope $P$ is the set of all polytopes which are combinatorially equivalent to $P$.
Making this definition rigorous requires to specify how $P$ is given exactly.
There are several natural choices, e.g., in terms of the coordinates of the vertices, the facets or both, which yield semialgebraic sets.
From work of Mn\"ev \cite{Mnev:1988} and Richter-Gebert~\cite{Richter-Gebert:1996} it is known that every basic semialgebraic set occurs as the realization space of a polytope in dimension at least four.
This is a precise way of saying: realization spaces of polytopes are complicated.
However, the situation is very different if we consider $3$-dimensional polytopes.
By Steinitz' theorem \cite[Section III]{Steinitz:1922} the combinatorial types of the $3$-polytopes correspond to the planar graphs which are $3$-connected; cf.\ \cite[Chapter~4]{Ziegler:1995}.
This entails that the realization spaces of $3$-polytopes are rather simple.
For instance, they are contractible, and each $3$-polytope admits a realization with rational coordinates.

The nontrivial core of Steinitz' result is the statement: each $3$-connected planar graph admits a realization of the vertex-edge graph of a $3$-polytope.
This lends itself to various strengthenings.
Among these are results of Koebe \cite{Koebe:1936}, Schramm \cite{Schramm:1992} and Springborn~\cite{Springborn:2005}, which can be summarized as follows; see also \cite[Theorem~4.13]{Ziegler:1995}.
\begin{theorem*}
  For every $3$-connected planar graph, there is a representation as the graph of a $3$-polytope whose edges are all tangent to the unit sphere $\Sph^2 \subset \RR^3$, and such that $0$ is the barycenter of the contact points.
  This representation is unique up to rotations and reflections of the polytope in~$\RR^3$.
\end{theorem*}
Interest in this line of research is motivated not only by polytope theory but also, e.g., by the geometry and topology of $3$-dimensional manifolds; see Thurston~\cite[Section 13.6]{Thurston:gt3m}.
The aforementioned result gives rise to \emph{Koebe realizations} of a $3$-polytope (with edges tangent to the sphere) and \emph{Springborn realizations} (which additionally require that the origin is the barycenter of the contact points).
These notions lead to constrained realization spaces, which again admit descriptions as semialgebraic sets defined by rational polynomials.
Basic model theory implies that there are Koebe and Springborn realizations such that the vertex and facet coordinates are real algebraic numbers; cf.\ \cite{Bochnak:1998}.
The purpose of this article is to study the resulting minimal degrees, which we call \emph{Koebe} and \emph{Springborn degrees}, respectively.

By constraining realization spaces of $3$-polytopes via imposing additional algebraic constraints we arrive at a class of interesting semialgebraic sets.
These are more rich than the full realization spaces of $3$-polytopes but still easier to understand than the infinitely more difficult realization spaces of $4$-polytopes.
One natural question is: Which $3$-polytopes admit a Koebe realization which is rational?
In general, this seems to be a hard problem.
The algorithmic approach to rational realizations of polytopes has been pioneered by Bokowski and Sturmfels in \cite{BokowskiSturmfels:1989}.
By reducing to oriented matroids, they show that deciding whether a polytope admits a realization with rational coordinates is equivalent to deciding whether a diophantine equation has a rational solution.
The latter is the rational formulation of Hilbert's 10th problem, which is still open; cf.\ \cite[Section 2.3]{BokowskiSturmfels:1989}.
Matiyasevich proved in \cite{Matiyasevich:1993} that the question over the integers has a negative answer, and there is no algorithm to decide whether a diophantine equation has an integer solution. 

Our contributions are the following.
Theorem~\ref{thm:triangular} gives an upper bound of the Koebe degree, which is doubly exponential in the number of vertices or facets.
The proof uses cylindrical algebraic decomposition (CAD), which is a method for quantifier elimination over real closed fields developed by Collins~\cite{CAD}; see also \cite[Section 5]{Basu:2006}.
Theorems~\ref{thm:Springbornbipyr} and \ref{thm:Koebebipyr} demonstrate that the Springborn and Koebe degrees are nontrivial invariants of a $3$-polytope: neither of them is bounded by any constant.
In our final result, Theorem~\ref{thm:stacked}, we show that stacked $3$-polytopes always admit a Koebe realization which is rational.
In a way the stacked polytopes may be considered the most simple class of convex polytopes.
So that result supports our intuition that the Koebe and Springborn degrees provide algebraic complexity measures which reflect a certain combinatorial complexity.
We close with a few open questions.

\paragraph{Acknowledgments.}
This work is inspired by discussions with Bernd Sturmfels on realizations of $3$-polytopes related to virus capsids, cf.\ \cite{TwarockLuque}.
We are grateful to Günter M. Ziegler for useful comments.

\section{Preliminaries}
\noindent
We will collect some useful facts about classical groups and how this is connected with polytopes in $\RR^3$.

\subsection{Lorentz transformations}  \label{sect:groups}
Consider the homogeneous quadratic form $q(x_0,x_1,x_2,x_3) = x_0^2 - x_1^2 - x_2^2 - x_3^2$ on $\RR^4$.
The pair $(\RR^4,q)$ is called \emph{Minkowski $4$-space.} 
This is preserved by the orthogonal group $\Orth_{1,3}\RR$ of \emph{Lorentz transformations}.
The set 
\begin{equation}\label{eq:sphere}
  \Sph^2 \ = \ \SetOf{x\in\RR^4}{x_0=1,\, q(x)=0}
\end{equation}
is the \emph{celestial sphere} of special relativity.
The Lorentz transformations which leave the set $I^{+}:=\smallSetOf{x\in\RR^4}{q(x)>0}$ invariant form a subgroup denoted as $\Orth_{1,3}^+\RR$. 
The subgroup of $\Orth_{1,3}^+\RR$ containing the matrices with positive determinant has index two, and we write it as $\SO_{1,3}^+\RR$. 
We have
\[\Orth_{1,3}^+\RR\ = \ \SO_{1,3}^+\RR \rtimes \langle\diag(1,-1,-1,-1)\rangle \enspace ,\]
where $\rtimes$ is the semidirect product.
Notice that via dehomogenization $I^+$ becomes the $3$-dimensional open unit ball $I$ whose boundary is the sphere $\Sph^2$.
The group $\SO_{1,3}^+\RR$ acts on $\Sph^2$ as a subgroup of projective transformations in~$\RR^3$, and $\diag(1,-1,-1,-1)$ is the antipodal map on $\Sph^2$.
We will now analyze this group action.

Let $\KK$ be a subfield of $\RR$.
Then its quadratic extension $\LL=\KK[i]=\KK[x]/(x^2+1)$ is a subfield of $\CC=\RR[i]$.
Here we let $i$ denote the imaginary unit, and $z\mapsto\overline{z}$ is the complex conjugation.
We write $\FF\cup\{\infty\}$ for the projective line over any field $\FF$.
The map $\sigma:\CC\cup\{\infty\}\to\Sph^2$ induced by
\begin{equation}\label{eq:stereo}
  (z,w) \mapsto \bigl(z\overline{z}+w\overline{w},\, z\overline{z}-w\overline{w},\, z\overline{w}+w\overline{z},\, -i(z\overline{w}-w\overline{z})\bigr)
\end{equation}
is the \emph{stereographic projection}.
The following is classical; see, e.g., \cite[Section 9.II.5]{Dieudonne:1971}.
\begin{lemma}\label{lem:equivalent}
  The group $\SO_{1,3}^+\KK$ is isomorphic to $\PSL_2 \LL$.
  In particular, for $\KK=\RR$ we obtain $\SO_{1,3}^+\RR \cong \PSL_2\CC$.
  Moreover, the natural action of $\SO_{1,3}^+\KK$ on $\Sph^2$ is equivalent to the natural action of $\PSL_2\LL$ on $\CC\cup\{\infty\}$.
\end{lemma}
We give a short proof for the sake of completeness.
\begin{proof}
  From $x=(x_0,x_1,x_2,x_3)\in\KK^4$ we define the Hermitian matrix
  \begin{equation}\label{eq:X}
    X \ = \ \begin{pmatrix} x_0+x_1 & x_2+ix_3 \\ x_2-ix_3 & x_0-x_1 \end{pmatrix} \enspace ,
  \end{equation}
  with entries in $\LL=\KK[i]$.
  We have $\det X = x_0^2 - x_1^2 -x_2^2 - x_3^2 = q(x)$.
  For $A\in\SL_2\LL$, with  $A^*:=\transpose{\bar{A}}$ its conjugate transpose, we have $\det(AXA^*) = \det X = q(x)$ as $\det A=\det A^*=1$.
  Moreover, $(AXA^*)^*= A^{**} X^* A^*=AXA^*$, i.e., $AXA^*$ is Hermitian.
  Since every Hermitian $2{\times}2$-matrix can be written in the form \eqref{eq:X}, the map
  \begin{equation}\label{eq:action}
    X\mapsto AXA^*
  \end{equation}
  defines a linear action on $\KK^4$.
  As we saw this preserves the quadratic form~$q$, whence we obtain an element of $\SO_{1,3}^+\KK$.
  The subgroup $\langle\pm\1\rangle$ forms the kernel of the resulting epimorphism from $\SL_2\LL$ onto $\SO_{1,3}^+\KK$.
  The quotient $\SL_2\LL/\langle-\1\rangle$ is the projective special linear group $\PSL_2\LL$.

  It remains to compare the actions.
  Pick $(z,w)\in\CC^2\setminus\{(0,0)\}$, and let $(x_0,x_1,x_2,x_3):=\sigma(z,w)$ via \eqref{eq:stereo}; this defines a Hermitian matrix $X$ as in~\eqref{eq:X}.
  We abbreviate the inverse map from Hermitian matrices to $\CC^4$ by $h$, i.e., $h(X)=(x_0,x_1,x_2,x_3)$.
  For $A\in\SL_2\LL$ we compute
  \[
    h(AXA^*) \ = \ \sigma\bigl(\, A\cdot\colvectwo{z}{w} \,\bigr) \enspace ,
  \]
  and this proves the claim.
\end{proof}

We write $\Sph^2_\KK$ for the intersection of $\Sph^2$ with $\KK^3$.
The stereographic projection induces a bijection onto the projective line $\LL\cup\{\infty\}$.
The group actions can be restricted.
In this way the action of $\SO_{1,3}^+\KK$ on $\Sph^2_\KK$ is equivalent to the action of $\PSL_2\LL$ on $\LL\cup\{\infty\}$.
The group $\PSL_2\LL$ is a normal subgroup of the group $\PGL_2\LL=\GL_2\LL/\LL^*\1$, and the quotient is given by the quadratic residues.
In particular, $\PSL_2\LL=\PGL_2\LL$ if and only if every element in $\LL$ is a square.
The action of $\PGL_2\LL$ on the projective line over $\LL$ is sharply triply transitive.
The elements of $\PGL_2\CC$ are called \emph{Möbius transformations}.

\subsection{Koebe realizations and hyperbolic geometry}
Let $P$ be a convex polytope of dimension $3$. 
Its vertex-edge graph will be denoted $\Gamma=(V,E)$. 
Let $\koebe\subset\RR^3$ be another polytope, combinatorially equivalent to $P$, such that the edges of $\koebe$ are tangent to the unit sphere $\Sph^2\subset\RR^3$.
Identifying $\RR^3$ with the affine subspace $x_0=1$ of $\RR^4$ allows to view $\Sph^2$ as the celestial sphere~\eqref{eq:sphere}.
We call the polytope $\koebe$ a \emph{Koebe realization} of $P$, and the \emph{contact points} are the points of tangency.
For an edge $e\in E$ we indicate the corresponding contact point with  $\tau_\koebe(e)$ and we set
\begin{equation}
  \beta_\koebe \ := \ \frac{1}{m} \sum_{e\in E} \tau_\koebe(e) \enspace ,
\end{equation}
where $m=\# E$ is the number of edges of $\koebe$.
The point $\beta_\koebe$ is the \emph{edge barycenter} of $\koebe$.

An admissible transformation in $\Orth_{1,3}^+\RR$ maps one Koebe realization of $P$ to some other Koebe realization.
In fact, all Koebe realizations can be obtained from one of them by applying an admissible transformation; see~\cite{Springborn:2005}.
Here a transformation is \emph{admissible} to $\koebe$ if no point of $\koebe$ is mapped to a point $(x_0,x_1,x_2,x_3)\in\RR^4$ with $x_0=0$.

Let us  recall some definitions and results from \cite{Springborn:2005}.
Fix a polytope $P$ with $m$ edges and a Koebe realization $\koebe$ of such a polytope.
We consider the open unit ball $I$ as the Klein model for the hyperbolic $3$-space, and $\Sph^2$ is its infinite boundary. 
A \emph{horosphere} in $I$ is a hyperbolic sphere whose center lies in the boundary $\partial I=\Sph^2$.
For a fixed horosphere $h$ in $I$ we define the function which sends a point $x\in I$ to
\[
  \delta_h(x) \ = \ 
  \begin{cases}
    -\textnormal{dist}(x,h)\textnormal{ if $x$ is inside $h$}\\
    0\textnormal{ if }x\in h\\
    \textnormal{dist}(x,h)\textnormal{ if $x$ is outside $h$}
  \end{cases}
\]
The \emph{point of minimal distance sum} of the Koebe realization $\koebe$ is the unique minimum in $I$  of the function
\begin{equation}
\label{fun:minspri}
x \mapsto \sum_{i=1}^m \delta_{h_i}(x) \enspace ,
\end{equation}
where $h_i$ is a fixed horosphere centered in the $i$th contact point of $\koebe$.
Bobenko and Springborn investigated the existence and uniqueness of Koebe realizations via variational principles~\cite{BobenkoSpringborn:2004}.
Using these techniques Springborn showed \cite{Springborn:2005} that the point of minimal distance sum of $\koebe$ is zero if and only if the edge barycenter $\beta_\koebe$ is zero.
This answered a question posed by Günter M. Ziegler.
The following is the crucial step.

\begin{proposition}\cite[Lemma 1]{Springborn:2005} \label{prop:center}
  There is a transformation $T \in\Orth_{1,3}^+\RR$, admissible with respect to $\koebe$, such that $\beta_{T\koebe}=0$.
  Moreover, if $\tilde{T}$ is another such transformation, then $\tilde{T}=RT$ where $R$ is a rotation in $\RR^3$.
\end{proposition}
Notice that a transformation $T$ in $\Orth_{1,3}^+\RR$ with $\beta_{T\koebe}=0$ is always admissible:
Otherwise, a point of $\koebe$ would be mapped to the far hyperplane, and the images of the contact points would be contained in a hemisphere; the latter contradicts $\beta_{T\koebe}=0$.
In this way we obtain an orientation-preserving admissible transformation which produces a Koebe realization of $P$ whose point of minimal distance sum is the origin.
We call such a Koebe realization of $P$ a \emph{Springborn realization}; this is also a Koebe realization whose edge barycenter is the origin.
Thanks to Proposition \ref{prop:center}, any two Springborn realizations of $P$ differ by rotations and reflections.

\section{Constrained realizations and their degrees} 
\noindent
We now present one model for the realization space of a given ($3$-dimensional) polytope $P$, and explain how the Koebe realizations fit in.
Let $n=\# V$ be the number of vertices of $P$.
Fixing a labeling of the vertices, we can see the set of all realizations of $P$ as a semialgebraic set in $\RR^{3\times n}$; this is what we call the \textit{realization space} of $P$.


\begin{remark} \label{rem:realspace}
  The defining equations and strict inequalities can be obtained from the combinatorics of $P$ as follows.
  Let  $(v_1,\dots,v_n)$ be the columns of $3{\times}n$-matrix of indeterminates. 
  We identify the vertices of $P$ with their labels $1,2,\dots,n$.
  For each facet $F$ of $P$, we pick three affinely independent vertices $i,j,k$ which lie in $F$.
  Then each other vertex, $\ell$, gives rise to one of two conditions.
  Either $\ell$ is contained in $F$, too, then we have
  \[
    \det\begin{pmatrix} 1 & 1 & 1 & 1 \\ v_{i} & v_{j} & v_{k} & v_\ell \end{pmatrix} \ = \ 0 \enspace ;
  \]
  or $\ell$ is not contained in $F$, and we have
  \[
    \epsilon \cdot \det\begin{pmatrix} 1 & 1& 1 & 1 \\ v_{i} & v_{j} & v_{k} & v_\ell \end{pmatrix} \ > \ 0 \enspace .
  \]
  The sign $\epsilon\in\{\pm 1\}$ depends on the orientation of $F$ induced by $i,j,k$, i.e., if the cross product of $v_j-v_i$ and $v_k-v_i$ determines the outer normal direction for the face $F$, then the sign is negative, otherwise it is positive.
\end{remark}

Within this realization space the set of Koebe realizations is constrained by the edge tangency conditions, which read
\begin{equation}\label{eq:koebe-constraint}
  (v_i \cdot (v_j-v_i))^2-\norm{v_i-v_j}^2(\norm{v_i}^2-1) \ = \ 0
\end{equation}
for any pair of adjacent vertices $i$ and $j$.
This also turns the \emph{Koebe realization space} of $P$ into a basic semialgebraic set.

\begin{remark}
  The realization space described in Remark \ref{rem:realspace} encodes a polytope via the coordinates of its vertices.
  Dually, it could also be written in terms of the facets.
  Various other models of realization spaces have been proposed in the literature:
  The space considered by Richter-Gebert \cite{Richter-Gebert:1996} factors out affine transformations.
  Rastanawi, Sinn and Ziegler analyze the \enquote{centered realization space} where a polytope is given by vertices \emph{and} facet normals \cite{Centered:2020}.
  Gouveia, Macchia, Thomas and Wiebe introduce a realization space based on slack matrices \cite{Slack:2019}.
  We refer to \cite[Section 6]{Centered:2020} for a comparison.
\end{remark}

\begin{lemma}\label{lem:welldef}
  Within the Koebe realization space of $P$ those realizations whose coordinates are all algebraic form a dense subset.
\end{lemma}
\begin{proof}
  This holds more generally.
  In fact, let $S(x_1,\dots,x_k)$ be a system of polynomials with rational coefficients and $\koebe=(q_1,\dots,q_k)$ any real solution. 
  Pick $Q'\in \QQ^k$ such that $\norm{\koebe-\koebe'}^2<\delta$ with $\delta$ rational and arbitrarily small. 
  The sentence
  \[ \exists X=(x_1,\dots,x_k) \,:\, S(X)=0  \land \norm{X-\koebe'}^2<\delta \]
  is true in $\RR$, and therefore it is also true in $\RR\cap\overline\QQ$, the field of real algebraic numbers.
  This is a consequence of the Tarski--Seidenberg Principle; see \cite[Proposition 5.2.3]{Bochnak:1998}.
\end{proof} 

This observation gives rise to our key definition.
Given any realization $R$ of $P$, we denote with $\QQ[R]$ the field extension of $\QQ$ given by the vertex coordinates of $R$.
The \emph{Koebe degree} $\kappa(P)$ is the minimal algebraic degree $\lvert\QQ[\koebe]:\QQ\rvert$ where $\koebe$ varies over the Koebe realizations of $P$.
Due to Lemma~\ref{lem:welldef} the Koebe degree is always finite.
The Koebe degree of $P$ equals one if and only if $P$ admits a Koebe realization with rational coordinates.
Recall that, without the extra edge tangency condition \eqref{eq:koebe-constraint}, $P$ always admits a rational realization as a consequence
of Steinitz’s theorem \cite{Ziegler:1995}.

\begin{example}\label{ex:simplex}
\begin{figure}
\centering
\includegraphics[scale=0.2]{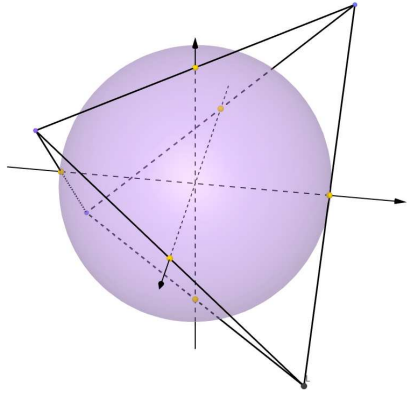}
\caption{A Springborn realization of the tetrahedron with contact points marked.}
\label{fig:simplex}
\end{figure}
  The four columns of the rational matrix
  \[
    \begin{pmatrix}
      -1 &  1 &  1 & -1 \\
       1 & -1 &  1 & -1\\
       1 &  1 & -1 & -1 \\
    \end{pmatrix}
  \]
  provide a Koebe realization of the tetrahedron. 
  Therefore its Koebe degree is one.
  The six contact points form the vertices of a regular octahedron, which sum to zero.
  Hence this is even a Springborn realization; see Figure~\ref{fig:simplex}.
\end{example}


\begin{lemma}\label{lem:interior}
  Let $\koebe$ be a Koebe realization of $P$ over the field $\KK=\QQ[\koebe]$.
  Then there is another Koebe realization $\koebe'$ of $P$, over the same field~$\KK$, such that the origin lies in the interior of $\koebe'$.
\end{lemma}
\begin{proof}
  Let $\koebe$ be a Koebe realization of $P$, with $\KK=\QQ[\koebe]$.
  By Proposition~\ref{prop:center} there is an admissible transformation in $\Orth^+_{1,3}\RR$ which maps $\koebe$ to a Koebe realization $\koebe''$, over $\RR$, such that the origin is the edge barycenter.
  The edge barycenter lies in the interior of any Koebe realization.
  As $\Orth^+_{1,3}\KK$ is a dense subgroup of $\Orth^+_{1,3}\RR$ we conclude that, for any $\epsilon>0$, there is a Koebe realization $\koebe'$ of $P$, over $\KK$, such that the distance of the edge barycenter of $\koebe'$ to the origin is less than $\epsilon$.
  It follows that $\koebe'$ is a Koebe realization of $P$, over $\KK$, and the origin lies in the interior, as desired.
\end{proof}

Let us assume that $P$ has a realization $P = \conv(V)$ as a convex hull in $\RR^3$ such that the origin is in the interior of $P$.
Then the \emph{polar polytope} is
\[
  P^\vee \ = \ \SetOf{ a \in \RR^3 }{ a \cdot v\leq  1 \ \text{for all} \ v \in V } \enspace .
\]
The face lattice of $P^\vee$ is anti-isomorphic to the face lattice of $P$.
A polytope is \emph{dual} to $P$ if it is combinatorially equivalent to $P^\vee$; we denote the combinatorial type as $P^*$; see \cite[Section 2.3]{Ziegler:1995}.

\begin{proposition}\label{prop:dual}
  The Koebe degrees of $P$ and its dual $P^*$ agree.
\end{proposition}
\begin{proof}
  Let $\koebe$ be a Koebe realization of $P$.
  In view of Lemma~\ref{lem:interior} we may assume that the origin lies in the interior of $\koebe$. 
  Notice that the polar $\koebe^\vee$ is a Koebe realization of the polar $P^\vee$: in $3$-space, the polar of an edge tangent to $\Sph^2$ is an edge which is again tangent.
  The vertices of the polar $\koebe^\vee$ correspond to the facets of $\koebe$.
  Since the facet coordinates of $\koebe$ can be derived from the vertices by solving systems of linear equations, the polar $\koebe^\vee$ is a Koebe realization of $P^*$ over the same field $\KK$.
  
  Applying this reasoning to a Koebe realization of minimal degree reveals that the Koebe degree of $P^*$ does not exceed the Koebe degree of $P$.
  Now the claim follows from the equality $P^{**}=P$.
\end{proof}

Our next goal is an explicit upper bound for the Koebe degree.
The following is the key ingredient; it may be seen as a sharpened version of Lemma~\ref{lem:welldef}.
For a field $\KK\subset\RR$ and a point $a$ in $\RR^N$ let $\KK[a]$ be the extension field over $\KK$ generated by the $N$ coordinates of $a$.

\begin{lemma}\label{lem:bound}
  Let $\KK\subset\RR$ be a field.
  Let $a$ be an isolated point of some real semialgebraic set in $\RR^N$ defined by polynomials of degree at most $D$ with coefficients in $\KK$.
  Then the degree of the field extension satisfies
  \[
    \lvert\KK[a]:\KK\rvert \ \leq \ (2D)^{2^{N}-1}2^{-N} \enspace.
  \]
\end{lemma}
\begin{proof}
  We apply the method of cylindrical algebraic decomposition (CAD) to the semialgebraic set in $\RR^N$; see \cite[Section 5]{Basu:2006}.
  This is an effective version of Tarski--Seidenberg \cite[Proposition 5.2.3]{Bochnak:1998}, which works by iterated projections or variable eliminations.
  In the step at level $k$, for $1\leq k\leq N$, we obtain a set of polynomials $S_k$ with coefficients in $\KK$ defining a semialgebraic decomposition of $\RR^k$ into cells $\cC_k$, which are the projections of cells in $\cC_{k+1}$; see \cite[Section 11.1.1]{Basu:2006}.
  Because the decomposition is compatible with the initial system, the polynomials defining our semialgebraic set have constant sign over each cell in $\RR^{N}$.
  Each isolated point $a$ of the semialgebraic set yields a distinct cell, and its projection at every step is a cell consisting of a single point in $\RR^k$.
  Let $d_k$ be the maximum degree with respect to the $k$th variable of the polynomials in $S_k$.
  The algebraic degree $\lvert\KK[a]:\KK\rvert$ is bounded by the product of the degrees $d_k$, and
  the complexity analysis in \cite[Section 11.1.1]{Basu:2006} establishes that $d_{k}\leq 2(d_{k+1})^2$. Therefore, we have 
  \[
  \lvert\KK[a]:\KK\rvert \ \leq \prod_{k=1}^{N} d_k \leq \prod_{k=1}^{N-1} 2^{2^{k-1}-1} \cdot D^{2^{k-1}}.
  \]
  Hence we arrive at the doubly exponential upper bound $(2D)^{2^N-1}2^{-N}$.
\end{proof}

The following is our first main result.
\begin{theorem} \label{thm:triangular}
  Let $P$ be a $3$-polytope with $n=\# V$ vertices and a triangular facet.
  Then its Koebe degree is at most
  \begin{equation}\label{eq:triangular}
    \kappa(P) \ \leq \ 8^{2^{3n-9}-1} \cdot 2^{-(3n-9)} \ < \ 2^{2^{3n}} \enspace.
  \end{equation}
\end{theorem}
\begin{proof}
  We label the vertices of the triangular face with $\{1,2,3\}$.
  Inside the realization space, we fix $v_1,v_2$ and $v_3$ in $\QQ^3$.
  There is at least one Koebe realization with these three vertices because $\SO_{1,3}^+\RR\cong\PSL_2\CC=\PGL_2\CC$ is triply transitive on the sphere $\Sph^2$, and $\QQ^2$ is dense in $\RR^2$.
  Thus we can find an admissible transformation.
  There are at most two such Koebe realizations because that group actually acts \emph{sharply} triply transitive on the sphere.
  We could have another realization coming from the whole group $\Orth_{1,3}^+\RR$, i.e., taking also the antipodal reflection into account.
  These finitely many Koebe realizations are cut out in $\RR^{3n-9}$ by rational equalities and inequalities of degree at most $4$. 
  Now Lemma~\ref{lem:bound} gives the desired bound for $N= 3n-9$ and $D=4$.
\end{proof}

If a $3$-polytope $P$ does not contain a triangular facet, then its dual $P^*$ does.
This follows from Euler's equation and double counting.
In that case, we get a bound for $\kappa(P^*)=\kappa(P)$ like \eqref{eq:triangular} by replacing $n$ with the number of facets of $P$.

Let us state a useful lemma which says that, for computing the Koebe degree, we can just look at contact points.

\begin{lemma}\label{lem:beta}
  Let $\KK\subset\RR$ be a field, and let $Q$ be some Koebe realization of~$P$.
  Then all the contact points $\tau_Q(e)$ lie in $\KK^3$ if and only if all the vertices lie in $\KK^3$.
  Consequently, in that case the edge barycenter $\beta_Q$ lies in $\KK^3$.
\end{lemma}

\begin{proof}
  Suppose that the vertices lie in $\KK^3$.
  Consider the edge $e$ linking the vertices $w$ and $v$.
  We define $\ell_{w}^2:=\norm{w}^2-1$, which is the length of the segment between $w$ and any of the contact points on the sphere $\Sph^2$ associated with the edges through $w$.
  Similarly for $\ell_v$.
  Both quantities lie in $\KK$.
  The contact point now satisfies
  \[
    \tau_Q(e) \ = \ \frac{w-v}{\ell_w+\ell_v}\,\ell_v+v \ = \ \frac{w-v}{(\ell_w+\ell_v)^2}\, (\ell_v^2+\ell_v\ell_w)+v
  \]
  and $2\ell_v\ell_w=\norm{w-v}^2-\norm{w}^2-\norm{v}^2+2$.
  The reverse implication follows from the description of the vertices as intersection points of the planes tangent to the sphere at the contact points. 
\end{proof}

\begin{remark}\label{rem:ratKoeb}
  Let $P$ be a $3$-polytope which admits a rational Koebe realization $\koebe$.
  If $\koebe'$ is a Koebe realization of $P$ with three rational contact points, then we know that $\koebe'$ is rational.
  It can be shown that there is an admissible rational transformation $T\in\PGL_4\QQ$ which maps $\koebe$ to $\koebe'$.
  This provides a direct way to decide whether a $3$-polytope has a rational Koebe realization or not.
  More generally, the Koebe degree is computable via CAD; cf.\ Lemma~\ref{lem:bound}.
  Note that the transformation $T$ does not need to lie in $\Orth_{1,3}^+\QQ$.
\end{remark}

Next we define the \textit{Springborn degree} $\sigma(P)$ as the minimal algebraic degree of any Springborn realization of $P$.
This is still well defined and finite thanks to Lemma \ref{lem:welldef} and its proof.
Yet the Springborn degree is somewhat better accessible than the Koebe degree.
One reason is given in the following lemma.

\begin{proposition}\label{prop:vol}
  Let $\koebe$ be a Springborn realization of $P$.
  Then the algebraic degree of the volume $\vol\koebe$ divides the Springborn degree $\sigma(P)$.
\end{proposition}
\begin{proof}
  First, the volume of a polytope $P$ can be written as a rational polynomial in the vertex coordinates; this follows, e.g., from triangulating $P$ without new vertices.
  Second, by Proposition~\ref{prop:center} any two Springborn realizations only differ by a linear isometry of $\RR^3$, which preserves the volume.
  This implies that the volume is in $\mathbb{Q}[\koebe]$, for every Springborn realization $\koebe$.
  Thus the claim.
\end{proof}

As in the proof of Proposition \ref{prop:dual} we may conclude that the Springborn degrees of $P$ and its dual agree, i.e., $\sigma(P)=\sigma(P^*)$.
Together with Example~\ref{ex:simplex}, the next two examples describe the Springborn degrees of all five Platonic solids.
\begin{example}\label{ex:octahedron}
  The six columns of the matrix
  \[
    \begin{pmatrix}
      \pm \sqrt{2} & 0 & 0 \\
      0 & \pm \sqrt{2} & 0 \\
      0 & 0 & \pm \sqrt{2}
    \end{pmatrix}
  \]
  provide a degree two Springborn realization of the octahedron.
  Its volume equals $\frac{8\sqrt{2}}{3}$, and so Proposition~\ref{prop:vol} shows that the Springborn degree of the octahedron equals two.
  Moreover, the six columns of the matrix
  \[
    \begin{pmatrix}
      -1 & 1 & 1 & x & -y & -y \\
      1 & -1 & 1 & -y & x & -y \\
      1 & 1 & -1 & -y & -y & x \\

    \end{pmatrix}
  \]
  where $y=\frac{1}{17}\sqrt{137+48\sqrt{2}}$  and $x=\frac{7-4\sqrt{2}}{17}$, give a degree four Koebe realization of the octahedron. This realization has three rational contact points. By Remark \ref{rem:ratKoeb}, the Koebe degree of the octahedron cannot be one, and therefore, it is equal to two.

  The same holds for the cube, which is dual to the octahedron.
\end{example}
\begin{example}
  The dodecahedron has a Springborn realization given by the 20 columns of the matrix
  \[
    \begin{pmatrix}
      0 & \pm \psi^{2} & \pm 1 & \pm \psi \\
      \pm 1 & 0 & \pm\psi^{2} & \pm \psi \\
      \pm \psi^{2} & \pm 1 & 0  & \pm \psi
    \end{pmatrix} \enspace ,
  \]
  where $\psi=\tfrac{2}{1+\sqrt{5}}$ is the reciprocal of the golden ratio.
  Since the squared norm of the vertices is not rational, there is no Springborn representative with rational coordinates.
  The Springborn degree of the dodecahedron and its dual, the icosahedron, equals two.
\end{example}

We now want to explore the relation between these two algebraic notions. 
Our next result relates the Koebe and Springborn degrees.

\begin{proposition}\label{prop:kob=sch}
  For a $3$-polytope $P$ with $m$ edges we have
  \[
    \kappa(P) \ \leq \ \sigma(P) \ \leq \ \frac{(2m+2)^{7}}{4} \cdot \kappa(P) \enspace .
  \]
\end{proposition}
\begin{proof}
  The first inequality is trivial, and we focus on the second one.
  To this end let $\koebe$ be a Koebe realization with $\KK=\QQ[\koebe]$ of minimal degree $\kappa(P)=\lvert\KK:\QQ\rvert$.
  We want to employ Proposition~\ref{prop:center} and provide an explicit version of the function \eqref{fun:minspri} on the Klein open $3$-ball $I$.
  With $x=(1,x_1,x_2,x_3)\in I$ and $x':=(x_1,x_2,x_3)$ this reads
  \[
    \sum_{e\in E}\log(1-\tau_Q(e)\cdot x')-\frac{m}{2}\log(1-\norm{x'}^2) \enspace .
  \]
  Since this is a convex and differentiable function on $I$, the unique minimum is characterized by the vanishing of the gradient.
  Consequently, the point of minimal distance sum is the unique common zero in $I$ of the three polynomials
  \[
    p_k(x) \ := \ mx_k\prod_{e\in E} (1-\tau_Q(e)\cdot x')-\sum_{e\in E} (\tau_Q(e))_k(1-\norm{x'}^2)\prod_{f\neq e}(1-\tau_Q(f)\cdot x')
  \]
  where $k=1,2,3$.
  That zero, $a$, is an isolated solution to a system of semialgebraic constraints, which is why Lemma \ref{lem:bound} applies.
  With the notation of Lemma \ref{lem:bound} we have $D\leq m+1$ and $N=3$, which entails
  \[
    \lvert\KK[a]:\KK\rvert \ \leq \ \frac{(2m+2)^{7}}{8} \enspace.
  \]
  We can find a transformation $T\in\Orth_{1,3}^+\KK\big[a,\sqrt{1-\norm{a}^2}\big]$ which brings the point of minimal distance sum of $\koebe$ to the origin.
  By the discussion preceding Proposition~\ref{prop:center} it follows that $T\koebe$ is a Springborn realization of $P$.
  Due to Lemma~\ref{lem:beta} its vertex coordinates lie in $\KK\big[a,\sqrt{1-\norm{a}^2}\big]$, and this gives the bound.
\end{proof}

In the sequel we write $B_k$ for the bipyramid over a $k$-gon, where $k\geq 3$.
We let $\phi$ denote Euler's totient function; i.e., $\phi(n)$ counts the positive integers up to $n$ that are relatively prime to $n$.
The case $k=4$ is the octahedron, and $\phi(4)=2$.
Recall that we found $\sigma(B_4)=2$ in Example~\ref{ex:octahedron}.
\begin{theorem}\label{thm:Springbornbipyr}
  Let $k\geq 4$ be an integer.
  If $k$ is odd, the Springborn degree $\sigma(B_k)$ is either $\frac{1}{2}\phi(k)$ or $\phi(k)$.
  If $k$ is even, the Springborn degree $\sigma(B_k)$ equals $\phi(k)$.
  In particular, the Springborn degree of a $3$-polytope is not bounded by any constant.
\end{theorem}
\begin{proof}
  An explicit Springborn realization of $B_k$ is given by 
  \begin{equation}\label{eq:Springbornbipyr}
    \renewcommand\arraystretch{2}
    \begin{pmatrix}
      \frac{\cos(2\pi/k)}{\cos(\pi/k)} & \cdots &  \frac{\cos(2k\pi/k)}{\cos(\pi/k)} & 0 \\
      \frac{\sin(2\pi/k)}{\cos(\pi/k)} & \cdots &  \frac{\sin(2k\pi/k)}{\cos(\pi/k)} & 0 \\
      0 & \cdots & 0 & \pm\frac{1}{\sin(\pi/k)} 
    \end{pmatrix}
  \end{equation}
  The volume of that realization of $B_k$ is given by 
  \[ V_k \ = \ \frac{2k}{3\cos(\pi/k)} \enspace.\]
  As reported in \cite{Watkins:1993}, the algebraic degree of $\cos(\pi/k)$ equals $\tfrac{1}{2}\phi(2k)$.
  Further, the realization \eqref{eq:Springbornbipyr} exists over the field $\QQ[\sin(\pi/k),\cos(\pi/k)]$. 
  Thanks to Proposition~\ref{prop:vol}, the Springborn degree $\sigma(B_k)$ is either $\tfrac{1}{2}\phi(2k)$ or $\phi(2k)$.
  Notice that $\phi(2k)=2\phi(k)$ for $k$ even, and $\phi(2k)=\phi(k)$ when $k$ is odd.
  
  Moreover, $\sigma(B_k)$ equals $\tfrac{1}{2}\phi(2k)$ if $\QQ[\cos(\pi/k),\sin(\pi/k)]=\QQ[\cos(\pi/k)]$.
  We claim the latter holds if and only if $k$ is even.
  Let $\zeta_{2k}$ be a $2k$th primitive root of unity.
  Since 
  \[ \lvert\QQ[\zeta_{2k},i]:\QQ\rvert \ = \ \lvert\QQ[\zeta_{2k},i]:\QQ[\zeta_{2k}]\rvert\cdot\phi(2k) \]
  and
  \[
    \begin{aligned}
     \lvert \QQ[\zeta_{2k},i]:\QQ\rvert \, &= \, \lvert\QQ[\zeta_{2k},i]:\QQ[\cos(\pi/k),\sin(\pi/k)]\rvert \cdot \lvert\QQ[\cos(\pi/k),\sin(\pi/k)]:\QQ\rvert \\
      &=\, 2\cdot \lvert\QQ[\cos(\pi/k),\sin(\pi/k)]:\QQ[\cos(\pi/k)]\rvert\cdot\lvert\QQ[\cos(\pi/k)]:\QQ\rvert \\
      &=\, 2\cdot\lvert\QQ[\cos(\pi/k),\sin(\pi/k)]:\QQ[\cos(\pi/k)]\rvert\cdot\frac{\phi(2k)}{2} \enspace,
    \end{aligned}
  \]
  to prove the claim we need to show that $i\in \QQ[\zeta_{2k}]$ if and only if $k$ is even.
  This follows as the roots of unity in $\QQ[\zeta_{2k}]$ are of order dividing $2k$.
  We conclude that $\sigma(B_k)=\tfrac{1}{2}\phi(2k)=\phi(k)$ if $k$ is even.
\end{proof}
\begin{figure}
\centering
\includegraphics[scale=0.2]{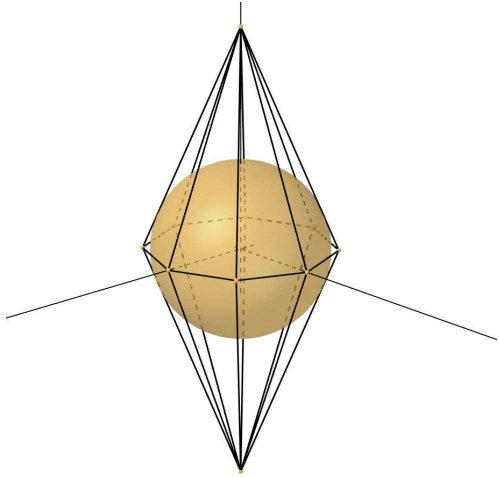}
\caption{A Springborn realization of the bipyramid $B_8$.}
\end{figure}

\begin{remark}
  From the explicit coordinate representation in~\eqref{eq:Springbornbipyr} we see that, in a Springborn realization, the $k$ \enquote{equatorial} vertices of the bipyramid $B_k$ form the vertices of a regular $k$-gon in the plane.
  Determining the algebraic degrees of the vertices of the regular polygons (and deciding their constructibilty with ruler and compass) is the topic of the final section of Gauss's \emph{Disquisitiones Arithmeticae} \cite[Section~366]{Disquisitiones}.
  For a recent account of the history of that famous result and its proof see \cite{AndersonChahalTop:2110.01355}.
\end{remark}

So far we lack lower bounds for the Koebe degrees, which seem somewhat harder to come by.
Yet for the bipyramids we are able to say something.
This is based on the following idea.
Four pairwise distinct complex numbers $a,b,c,d\in\CC$ define the \emph{cross ratio}
\begin{equation}\label{eq:crossratio}
  (a,b;c,d) \ = \ \frac{(c-a)(d-b)}{(c-b)(d-a)} \enspace ,
\end{equation}
which lies in $\CC\setminus\{0,1\}$, and which is invariant under the action of $\PGL_2\CC$ on the complex projective line $\CC\cup\{\infty\}$.
The cross ratio is invariant with respect to homogenization, which is why it also works with homogeneous coordinates.
Identifying the latter with $\Sph^2$ by stereographic projection makes this applicable to the contact points of any four edges in a Koebe realization.
Up to a rotation we may assume that none of the contact points is the north pole $\infty$ of that projection.

\begin{lemma}\label{lem:crossinv}
  Let $\koebe$ be a Koebe realization of $P$ with contact points in $\KK=\QQ[\koebe]$.
  Then the cross ratio of any four contact points lies in $\LL=\KK[i]$, and the degree of its real and imaginary part over $\QQ$ is a lower bound for $\kappa(P)$.
\end{lemma}
\begin{proof}
  The cross ratio \eqref{eq:crossratio} is a rational function in its parameters.
  The claim follows from the invariance with respect to the action of $\PGL_2\CC$, Lemma~\ref{lem:beta} and Proposition~\ref{prop:center}.
\end{proof}

We are now ready to prove the Koebe degree analog to Theorem~\ref{thm:Springbornbipyr}.
\begin{theorem}\label{thm:Koebebipyr}
  Let $k\geq 4$ be an integer.
  Then the Koebe degree $\kappa(B_k)$ is at least $\phi(k)/4$.
  In particular, the Koebe degree of a $3$-polytope is not bounded by any constant.
\end{theorem}
\begin{proof}
  We consider the Koebe realization \eqref{eq:Springbornbipyr} of the bipyramid $B_k$.
  The four columns of the matrix
  \[
    \begin{pmatrix}
      \cos(\pi/k) & \cos(3\pi/k) &  \cos(5\pi/k) & \cos(7\pi/k) \\
      \sin(\pi/k) & \sin(3\pi/k) & \sin(5\pi/k) & \sin(7\pi/k) \\
      0 & 0 & 0 & 0  
    \end{pmatrix}
  \]
  are contact points, written as images under the stereographic projection from $(0,0,1)$ in homogeneous coordinates.
  Their cross ratio is the real number 
  \[
    \alpha \ = \ \frac{(\zeta_k^2-1)^2}{(\zeta_k-1)(\zeta_k^3-1)} \enspace ,
  \]
  where $\zeta_k$ is the $k$th root of unity $e^{\frac{2i\pi}{k}}$. 
  We have
  \[
    \lvert\QQ(\alpha):\QQ\rvert \ = \frac{\lvert\QQ(\zeta_k):\QQ\rvert}{\lvert\QQ(\zeta_k):\QQ(\alpha)\lvert} \enspace ,
  \]
  and $\zeta_k$ is a root of the polynomial
  $
    \alpha(x-1)(x^3-1)-(x^2-1)^2
  $,
  which is non-zero because the coefficient of $x^2$ does not vanish.
  Therefore, 
  \[
    \lvert\QQ(\alpha):\QQ\rvert \ \geq \ \frac{\phi(k)}{4} \enspace ,
  \]
  and Lemma \ref{lem:crossinv} proves our claim.
\end{proof}

\begin{remark}
  The group $\SO_{1,3}^+\RR$ is a six-dimensional real Lie group, which acts faithfully on Minkowski $4$-space.
  It does not act on the Koebe realization space (of a given $3$-polytope) because of the admissibility conditions.
  However, it does act (faithfully and transitively) on a slightly larger space, which would take unbounded realizations into account.
  Since each Koebe realization $\koebe$ is a compact subset of $\RR^3$, any Lorentz transformation which moves $\koebe$ by at most $\epsilon>0$ in the Hausdorff distance is admissible.
  It follows that the dimension of the Koebe realization space equals $\dim\SO_{1,3}^+\RR=6$.
  This recovers Schramm's dimension count in \cite[Theorem~1.2]{Schramm:1992}.
  Similarly, the $3$-dimensional compact Lie group $\Orth_3\RR$ acts faithfully and transitively on the Springborn realization space, whence its dimension is three.
\end{remark}

\section{Stacked polytopes}
\noindent
A stacked $d$-polytope is a simplicial polytope obtained by starting with a $d$-simplex and successively adding vertices beyond a facet, see \cite[Section 3]{Ziegler:1995}.
In this section we look at Koebe realizations of stacked $3$-polytopes. 
Given three affinely independent points $v_1,v_2$ and $v_3$, we will denote by $(v_1,v_2,v_3)^+$ the half-space cut in $\RR^3$ by the affine hyperplane passing through $v_1,v_2$ and $v_3$ and containing the smaller portion of the Klein ball $I$. If both portions are equal, we pick either one. 

\begin{lemma}\label{lem:3dimsim}
  Let $v_1,v_2$ and $v_3$ be three distinct points in $\QQ^3$.
  Suppose that the edges of the triangle $(v_1,v_2,v_3)$ are tangent to the sphere.
  There exists a Koebe realization of the $3$-simplex containing these three points as vertices and such that the fourth point is contained in the half space $(v_1,v_2,v_3)^+$. 
  That realization has degree one. 
\end{lemma}

\begin{proof}
First, let us proof that such a realization exists. Take a point on the sphere contained in $\RR^3\setminus(v_1,v_2,v_3)^+$. Under stereographic projection, the three vertices determine three circles given by the tangency points of lines passing through the vertices and tangent to the sphere. Then, if we take the Soddy circle \cite{Eppstein2001TangentSA} of these three circles and project it back to the sphere, this is not a maximal circumference and uniquely determines a fourth vertex $v_4$ in $(v_1,v_2,v_3)^+$, which gives a Koebe realization.


 We set $\ell_i:=\sqrt{\norm{v_i}^2-1}$ for $i=1,2,3$.
  This is the length of the segment between $v_i$ and any of the contact point on the sphere $\Sph^2$ associated with the edges linking $v_i$ to its neighbors.
  The formula \cite[Theorem 7.2(d)]{hajja2006coincidences} refers to simplices of arbitrary dimension.
  Specializing to dimension three  yields
  \[
    r^2 \ = \ \frac{2}{(\sum_{i=1}^{3}\frac{1}{\ell_i})^2-\sum_{i=1}^{3}\frac{1}{\ell_i^2}} \enspace,
  \]
  where $r$ is the radius of the circle inscribed in the triangle $(v_1,v_2,v_3)$.

  Let $v_4$ be the fourth vertex of a Koebe realization of the $3$-simplex, and we know that such a $v_4$ exists thanks to the initial discussion.
  With this we get a value $\ell_4$ also for $v_4$.
  Using the formula \cite[Theorem 7.2(d)]{hajja2006coincidences} again but for four points yields
  \[
    4 \ = \ \frac{2}{r^2}-\frac{1}{\ell_4^2}+2\Big(\sum_{i=1}^3\frac{1}{\ell_i}\Big)\frac{1}{\ell_4}-\sum_{i=1}^3\frac{1}{\ell_i^2} \enspace.
  \]
  In particular, we have
  \[
    \frac{1}{\ell_4} \ = \ \frac{\sum_{i=1}^3\frac{1}{\ell_i}\pm\sqrt{\Delta}}{-1} \enspace,
  \]
  where
  \[
    \Delta \ = \ \Big(\sum_{i=1}^3\frac{1}{\ell_i}\Big)^2+\Big(\frac{2}{r^2}-4-\sum_{i=1}^3\frac{1}{\ell_i^2}\Big) \ = \ \frac{2}{r^2}+\frac{2}{r^2}-4 \ = \ \frac{4}{r^2}(1-r^2) \enspace.
  \]
  Observe that this discriminant is always non-negative. 

  We want to prove that $\ell_4^2$ is rational.
  As $\ell_i^2=\norm{v_i}^2-1$ and $\ell_i\ell_j=\tfrac{1}{2}(\norm{v_i-v_j}^2-\ell_i^2-\ell_j^2)$ are rational, it remains to check the rationality of
  \begin{equation}\label{eq:to-check}
    \sqrt{\Delta}\cdot\sum_{i=1}^3\frac{1}{\ell_i} \enspace .
  \end{equation}
  We have
  \[
    \sqrt{\Delta}\sum_{i=1}^3\frac{1}{\ell_i} \ = \ \frac{2}{r}\sqrt{1-r^2}\,\frac{\ell_1\ell_2+\ell_1\ell_3+\ell_2\ell_3}{\ell_1^2\ell_2^2\ell_3^2}\,\ell_1\ell_2\ell_3
  \]
  and $\ell_1\ell_2\ell_3 = (\ell_1+\ell_2+\ell_3) r^2$.
  So, for verifying the rationality of \eqref{eq:to-check} it suffices to check that
  \[
    r\, \sqrt{1-r^2}\, (\ell_1+\ell_2+\ell_3) \ \in \ \QQ \enspace .
  \]
  Yet this quantity is a rational multiple of the volume of the simplex with vertices $v_1,v_2,v_3$ and the origin, which is therefore rational.
  This shows that $\ell_4^2$ and $\norm{v_4}^2$ are rational.

  The point $v_4$ is the intersection $x$ of the three spheres of center $v_i$ and radius $\ell_i+\ell_4$, for $i=\{1,2,3\}$ and we can write it as solution of a linear system.
  The equality $\norm{x}=\norm{v_4}$ leads to the system
  \[
    \begin{cases}
      (x-v_1)\cdot(x-v_1) \ = \ (\ell_1+\ell_4)^2 \\
      (x-v_2)\cdot(x-v_2) \ = \ (\ell_2+\ell_4)^2 \\
      (x-v_3)\cdot(x-v_3) \ = \ (\ell_3+\ell_4)^2
    \end{cases}
  \]
  of quadratic equations.
  This can be transformed to
  \[
    \begin{cases}
      -2 x\cdot v_1 \ = \ (\ell_1+\ell_4)^2-(\ell_1^2+1)-(\ell_4^2+1) \\
      -2 x\cdot v_2 \ = \ (\ell_2+\ell_4)^2-(\ell_2^2+1)-(\ell_4^2+1) \\
      -2 x\cdot v_3 \ = \ (\ell_3+\ell_4)^2-(\ell_3^2+1)-(\ell_4^2+1) \enspace ,
    \end{cases}
  \]
  which is a system of rational linear equations with a unique solution.
  It follows that $x=v_4$ is rational.
\end{proof}

\begin{figure}\centering
\includegraphics[scale=0.3]{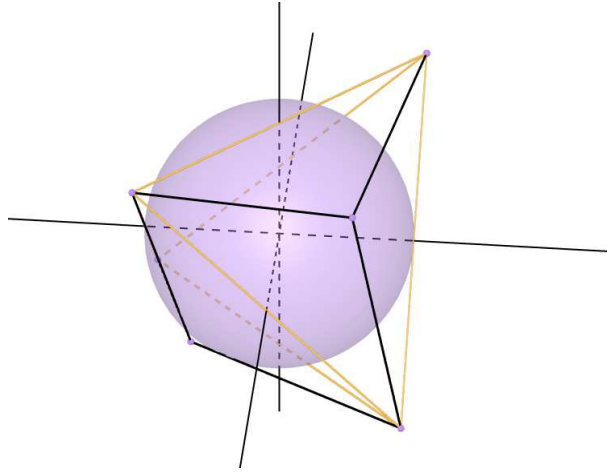}
\caption{A Koebe realization of a stacked polytope, constructed starting from the yellow tetrahedron stacking two more simplices. }
\end{figure}


We will call the Koebe realization found in the previous Lemma the \emph{small} simplex on $v_1,v_2$ and $v_3$.
Equipped with these observations, we are ready to prove our final result.
\begin{theorem}\label{thm:stacked}
  The Koebe degree of any stacked $3$-polytope equals one.
\end{theorem}
\begin{proof}
  A stacked polytope is defined recursively by starting with a simplex and repeated stackings over facets.
  Our proof will follow this inductive process, and it produces exactly one Koebe realization for each intermediate polytope simultaneously.

  We start with the Koebe representation of the tetrahedron in Example~\ref{ex:simplex}.
  Since it is even a Springborn representation it contains the origin in the interior.
  By picking the small simplex Koebe realization on three vertices of this Springborn representation, we can perform the first stacking using Lemma~\ref{lem:3dimsim}, and we arrive at a rational Koebe realization of the bipyramid~$B_3$.

  This can now be repeated for all remaining stackings.
  The only issue left is to make sure that we maintain convexity globally.
  In fact, the unique new vertex of a stacking lies beyond the triangular facet which is stacked over:
  This follows as the new vertex agrees with the intersection of the three new facets added by the stacking; that vertex is rational by Lemma~\ref{lem:3dimsim}.
  None of the new facets is tangent to the sphere, but each one of them contains a tangent edge.
  Consequently, they provide strict inequalities for the vertices of each intermediate stacked polytope, and so we get the next rational Koebe realization.
  None of the new facet hyperplanes passes through the origin, which is why we can continue.
  This finishes the proof.
\end{proof}

\section{Concluding remarks and open questions}
\noindent
Finding the Koebe and Springborn degrees of a $3$-polytope can be a challenging task.
Lemma \ref{lem:crossinv} provides a lower bound of the Koebe degree in terms of the cross ratios of four contact points.
It is an intriguing question whether this could lead to an algorithm.
\begin{question}
  Can the Koebe degree of a $3$-polytope be determined from the degree of the cross ratios of the contact points of some Koebe realization?
\end{question}

Moreover, we currently do not have an example of a polytope for which the two degrees differ.
Therefore, we ask the following. 
\begin{question}
  Given a polytope $P$, what is the precise relationship between its Koebe degree $\kappa(P)$ and its Springborn degree $\sigma(P)$?
  Does $\kappa(P)$ divide $\sigma(P)$?
  Or are they even equal?
\end{question}

A $d$-polytope in $\RR^d$ is called \emph{inscribed} if its vertices lie on the unit sphere, and it is \emph{circumscribed} if its facets are tangent.
More generally, a polytope is \emph{$k$-midscribed} if the tangency condition is satisfied by all faces of dimension~$k$.
These notion have been studied for polytopes of arbitrary dimension~$d$; see Padrol and Ziegler~\cite{PadrolZiegler:2016} for a survey with intriguing questions and conjectures.
In dimension $d=3$ a polytope is $1$-midscribed if at only if it is a Koebe realization.
\begin{question}
  What can be said about algebraic degrees of realizations of $k$-midscribed $d$-polytopes for arbitrary $d$ and $0\leq k\leq d-1$?
\end{question}
It is known that the realization spaces of polytopes constrained by tangency conditions may be empty.
For instance, Hodgson, Rivin and Smith~\cite{HodgsonRivinSmith:1993} characterize the inscribed $3$-polytopes.


\bibliographystyle{amsplain}
\bibliography{main.bib}

\providecommand{\bysame}{\leavevmode\hbox to3em{\hrulefill}\thinspace}
\providecommand{\MR}{\relax\ifhmode\unskip\space\fi MR }
\providecommand{\MRhref}[2]{%
  \href{http://www.ams.org/mathscinet-getitem?mr=#1}{#2}
}
\providecommand{\href}[2]{#2}
\begin{thebibliography}{10}

\bibitem{AndersonChahalTop:2110.01355}
Laura Anderson, Jasbir~S. Chahal, and Jaap Top, \emph{The last chapter of the
  {D}isquisitiones of {G}auss}, 2021, Preprint \arXiv{2110.01355}.

\bibitem{Basu:2006}
Saugata Basu, Richard Pollack, and Marie-Fran{\c{c}}oise Roy, \emph{Algorithms
  in real algebraic geometry}, vol.~36, Springer Berlin, Heidelberg, 2006.

\bibitem{BobenkoSpringborn:2004}
Alexander~I. Bobenko and Boris~A. Springborn, \emph{Variational principles for
  circle patterns and {K}oebe's theorem}, Trans. Amer. Math. Soc. \textbf{356}
  (2004), no.~2, 659--689.

\bibitem{Bochnak:1998}
Jacek Bochnak, Michel Coste, and Marie-Fran{\c{c}}oise Roy, \emph{Real
  algebraic geometry}, vol.~36, Springer-Verlag, Berlin Heidelberg, 1998.

\bibitem{BokowskiSturmfels:1989}
J\"{u}rgen Bokowski and Bernd Sturmfels, \emph{Computational synthetic
  geometry}, Lecture s in Mathematics, vol. 1355, Springer-Verlag, Berlin,
  1989.

\bibitem{CAD}
George~E. Collins, \emph{Quantifier elimination for real closed fields by
  cylindrical algebraic decomposition}, Automata theory and formal languages
  ({S}econd {GI} {C}onf., {K}aiserslautern, 1975), Lecture Notes in Comput.
  Sci., Springer, Berlin, 1975, pp.~134--183.

\bibitem{Dieudonne:1971}
Jean~A. Dieudonn\'{e}, \emph{La g\'{e}om\'{e}trie des groupes classiques},
  Springer-Verlag, Berlin-New York, 1971.

\bibitem{Eppstein2001TangentSA}
David Eppstein, \emph{Tangent spheres and triangle centers}, Amer. Math.
  Monthly \textbf{108} (2001), no.~1, 63--66.

\bibitem{Disquisitiones}
Carl~Friedrich Gauss, \emph{Disquisitiones {A}rithmeticae}, Springer-Verlag,
  New York, 1986, Translated and with a preface by Arthur A. Clarke, Revised by
  William C. Waterhouse, Cornelius Greither and A. W. Grootendorst and with a
  preface by Waterhouse.

\bibitem{Slack:2019}
Joao Gouveia, Antonio Macchia, Rekha~R. Thomas, and Amy Wiebe, \emph{The slack
  realization space of a polytope}, SIAM J. Discrete Math. \textbf{33} (2019),
  no.~3, 1637--1653.

\bibitem{hajja2006coincidences}
Mowaffaq Hajja, \emph{Coincidences of centers of edge-incentric, or balloon,
  simplices}, Results in Mathematics \textbf{49} (2006), no.~3-4, 237--263.

\bibitem{HodgsonRivinSmith:1993}
Craig~D. Hodgson, Igor Rivin, and Warren~D. Smith, \emph{A characterization of
  convex hyperbolic polyhedra and of convex polyhedra inscribed in the sphere},
  Bull. Amer. Math. Soc. (N.S.) \textbf{27} (1992), no.~2, 246--251, Erratum:
  ibid. {\bf 28} (1993), no. 1, 213.

\bibitem{Koebe:1936}
Paul Koebe, \emph{Kontaktprobleme der konformen {A}bbildung}, {Ber. S\"achs.
  Akad. Wiss. Leipzig, Math.-phys. Kl.} \textbf{88} (1936), 141--164.

\bibitem{Matiyasevich:1993}
Yuri~V. Matiyasevich, \emph{Hilbert's tenth problem}, Foundations of Computing
  Series, MIT Press, Cambridge, MA, 1993.

\bibitem{Mnev:1988}
N.~E. Mn\"{e}v, \emph{The universality theorems on the classification problem
  of configuration varieties and convex polytopes varieties}, Topology and
  geometry---{R}ohlin {S}eminar, Lecture Notes in Math., vol. 1346, Springer,
  Berlin, 1988, pp.~527--543.

\bibitem{PadrolZiegler:2016}
Arnau Padrol and G\"{u}nter~M. Ziegler, \emph{Six topics on inscribable
  polytopes}, Advances in discrete differential geometry, Springer, Berlin,
  2016, pp.~407--419.

\bibitem{Centered:2020}
Laith Rastanawi, Rainer Sinn, and G\"unter~M. Ziegler, \emph{On the dimensions
  of the realization spaces of polytopes}, Mathematika \textbf{67} (2021),
  no.~2, 342--365.

\bibitem{Richter-Gebert:1996}
J\"{u}rgen Richter-Gebert, \emph{Realization spaces of polytopes}, Lecture
  Notes in Mathematics, vol. 1643, Springer-Verlag, Berlin, 1996.

\bibitem{Schramm:1992}
Oded Schramm, \emph{How to cage an egg}, Selected works of {O}ded {S}chramm.
  {V}olume 1, 2, Sel. Works Probab. Stat., Springer, New York, 2011,
  pp.~87--104.

\bibitem{Springborn:2005}
Boris~A. Springborn, \emph{A unique representation of polyhedral types.
  {C}entering via {M}\"{o}bius transformations}, Math. Z. \textbf{249} (2005),
  no.~3, 513--517.

\bibitem{Steinitz:1922}
Ernst Steinitz, \emph{{E}ncyklop\"adie der mathematischen {W}issenschaften},
  ch.~{P}olyeder und {R}aumteilungen, pp.~1--139, Druck und Verlag von B.G.
  Teubner, Leipzig, 1914--1931.

\bibitem{Thurston:gt3m}
William~P. Thurston, \emph{Three-dimensional geometry and topology. {V}ol. 1},
  Princeton Mathematical Series, vol.~35, Princeton University Press,
  Princeton, NJ, 1997.

\bibitem{TwarockLuque}
Reidun Twarock and Antoni Luque, \emph{Structural puzzles in virology solved
  with an overarching icosahedral design principle}, Nature Communications
  \textbf{10} (2019), no.~4414.

\bibitem{Watkins:1993}
William Watkins and Joel Zeitlin, \emph{The minimal polynomial of
  {$\cos(2\pi/n)$}}, Amer. Math. Monthly \textbf{100} (1993), no.~5, 471--474.

\bibitem{Ziegler:1995}
G\"{u}nter~M. Ziegler, \emph{Lectures on polytopes}, Graduate Texts in
  Mathematics, vol. 152, Springer-Verlag, New York, 1995.

\end{thebibliography}
\end{document}